\documentclass[twoside]{amsart}
\usepackage{amssymb,amsmath,amscd}
\usepackage{amsthm}
\usepackage{amsfonts}
\usepackage{enumerate}
\usepackage{amsbsy}
\usepackage{geometry}
\usepackage[utf8]{inputenc}
\usepackage[T1]{fontenc}

\usepackage{microtype}

\newtheorem{theorem}{Theorem}[section]

\newtheorem{definition}[theorem]{Definition}

\newtheorem{example}[theorem]{Example}

\numberwithin{equation}{section}
\usepackage{caption} 
\usepackage{subcaption} 
\usepackage{graphicx}
\usepackage{pgfplots}
\usepackage[all]{nowidow}
\usepackage[utf8]{inputenc}
\usepackage{tikz}
\usepackage[inline]{enumitem}
\usepackage{blindtext}
\usepackage{amssymb}		
\usepackage{ulem}
\usepackage{amsfonts}

\usepackage[hidelinks,colorlinks=true,linkcolor=blue,citecolor=red, urlcolor=blue]{hyperref}
\numberwithin{equation}{section}
\pgfplotsset{compat=1.18}

\begin{document}
\title[Riesz Representation Theorems and Probability]{From Riesz to Kakutani: Representation Theorems and the Analytical Foundations of Probability}

\author[Antonio Cedeño-Pérez]{Luis Antonio Cedeño-Pérez}
\address{Departamento de Matemáticas, Facultad de Ciencias, Universidad Nacional Autónoma de México, Mexico City, Mexico}
\curraddr{}
\email{luisacp@ciencias.unam.mx}
\thanks{}

\author[Hugo Reyna-Castañeda]{Hugo Guadalupe Reyna-Castañeda}
\address{Departamento de Matemáticas, Facultad de Ciencias, Universidad Nacional Autónoma de México, Mexico City, Mexico}
\curraddr{}
\email{hugoreyna46@ciencias.unam.mx}
\thanks{}

\author[Mara Sandoval-Romero]{Mar\'ia de los \'Angeles Sandoval-Romero}
\address{Departamento de Matemáticas, Facultad de Ciencias, Universidad Nacional Autónoma de México, Mexico City, Mexico}
\curraddr{}
\email{selegna@ciencias.unam.mx}
\thanks{}

\begin{abstract}
  The analytical foundations of modern probability trace back to a sequence of representation theorems that reshaped functional analysis in the twentieth century. From Fréchet’s identification of linear functionals with vectors in Hilbert spaces to Kakutani’s characterization of measures on spaces of continuous functions, each theorem reveals how linearity, duality, and measure intertwine. Following this historical and conceptual path --—from Fréchet–Riesz to Riesz–Stieltjes, from $L^p$ duality to Riesz–Markov–Kakutani—-- we show that expectation, distribution, conditional expectation, and the Wiener measure are analytic manifestations of a single principle of representation. Viewed through this lens, probability theory appears not merely as an extension of measure theory, but as the geometric realization of functional analysis itself: every probabilistic notion embodies an existence-and-uniqueness principle in a space of functions.
\end{abstract}

\maketitle

\tableofcontents

\setlength{\parskip}{0.25em} 
\section{Introduction}
\sloppy

Modern probability theory rests largely upon the principles of functional analysis; cf. \cite{Gorostiza2001}. The idea that random phenomena can be described through linear objects---functionals, operators, or measures---finds its foundation in the representation theorems, which establish a deep correspondence between the structure of function spaces and the measures acting on them. In this sense, probability can be viewed as a natural extension of the notions of duality, continuity, and projection developed within the framework of twentieth-century functional analysis (see, for example, \cite{Dieudonne1981}).

Among the results that embody this connection are the representation theorems of Fréchet--Riesz, Riesz--Stieltjes, Riesz in $L^{p}$, and Riesz--Markov--Kakutani. Each of them reveals a different facet of a single unifying principle: the identification of every continuous linear functional with a measurable object ---whether a vector, a function, or a measure--- that represents it uniquely.

In this article we adopt this principle as the central viewpoint and show that several fundamental notions of probability ---expectation, the distribution of a random variable, conditional expectation, and the Wiener measure--- admit a unified analytical interpretation. Our contribution is to make explicit that each of these constructions arises as the solution of a linear representation problem naturally posed on an appropriate space of functions. Rather than treating these notions as isolated or purely probabilistic, we reveal that their existence and uniqueness follow directly from the Riesz representation theorems once the correct functional-analytic framework is identified. Although this perspective is implicitly present across the classical literature, it has not been articulated systematically: here we isolate the underlying linear functional behind each probabilistic concept and apply the corresponding Riesz theorem as the mechanism that produces it. This approach highlights the functional-analytic essence of these notions and places them within the historical development of the representation theorems of Fréchet, Riesz, Markov, and Kakutani.

From this viewpoint, the representation theorems are not merely technical instruments but constitute the conceptual architecture that supports the analytical formulation of modern probability. The journey from Riesz to Kakutani thus traces a line of continuity between functional analysis and measure theory, revealing that behind every essential probabilistic notion lies, silently, a statement of existence and uniqueness within a space of functions.

In the next section, we present a brief historical overview of the representation theorems of Fréchet--Riesz, Riesz--Stieltjes, Riesz in $L^{p}$, and Riesz--Markov--Kakutani. This review is not merely contextual: all these results are rooted in a simple geometric intuition from Euclidean spaces whose development, when extended to infinite-dimensional and topological settings, reveals the conceptual continuity that guides the theory. Observing this evolution shows how an elementary idea ultimately shapes the analytical architecture of modern probability.

\section{From Fréchet--Riesz to Kakutani: The Birth of the Representation Principle}

The history of representation theorems arises from a fundamental geometric intuition of linear algebra. In a Euclidean space $\mathbb{R}^N$, every linear functional $L:\mathbb{R}^N \to \mathbb{R}$ can be written as an inner product:
$$
L(x)=v_1x_1+\cdots+v_Nx_N=\langle v,x\rangle_{\mathbb{R}^N}=v^{\top}x.
$$

This familiar expression from introductory mathematics contains a deep idea: \textit{every linear functional can be represented by a vector} that acts as its direction of projection.

The row vector $v^{\top}$ stands for the functional, and the action of $L$ on $x$ simply measures how aligned they are (see Fig. \ref{figura-proyeccionRN}).

\begin{figure}[!ht]
\centering
\begin{tikzpicture}[scale=1.125]
\draw[-, semithick] (-0.5,0) -- (2.2,0);
\draw[-, semithick] (0,-0.5) -- (0,2.2);
\draw[-, thick] (0,0)--(1.55,1.55) node[left]{$v$};
\draw[-, thick] (0,0)--(1.45,0.45);
\draw[-, dashed, semithick] (0,0)--(-0.5,-0.5);
\draw (1.35,0.415) node[below]{$x$};
\draw[-, dashed, semithick] (1.45,0.45)--(1,1);
\end{tikzpicture}
\caption{The value of $L(x)$ measures the projection of the vector $x$ onto the direction of $v$.}
\label{figura-proyeccionRN}
\end{figure}
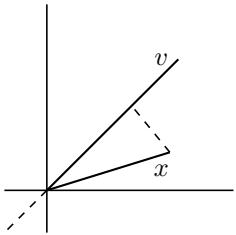

At the beginning of the twentieth century, analysis was undergoing a decisive transformation. The advances of Cauchy, Weierstrass, and Lebesgue had prepared the ground for the emergence of abstract structures: metric, normed, and topological spaces; cf. \cite{Bombal2003,Dieudonne1981}.

Within this new language, mathematicians began to ask whether the correspondence between functionals and vectors could persist beyond the finite-dimensional world. In other words: \textit{is it still true that every continuous functional can be seen as a projection in some direction when the space has infinitely many dimensions?}

The answer emerged gradually, through a chain of discoveries that outlined the birth of modern functional analysis and what we now call the Riesz representation theorems.

\begin{figure}[!ht]
\centering
\includegraphics[scale=0.125]{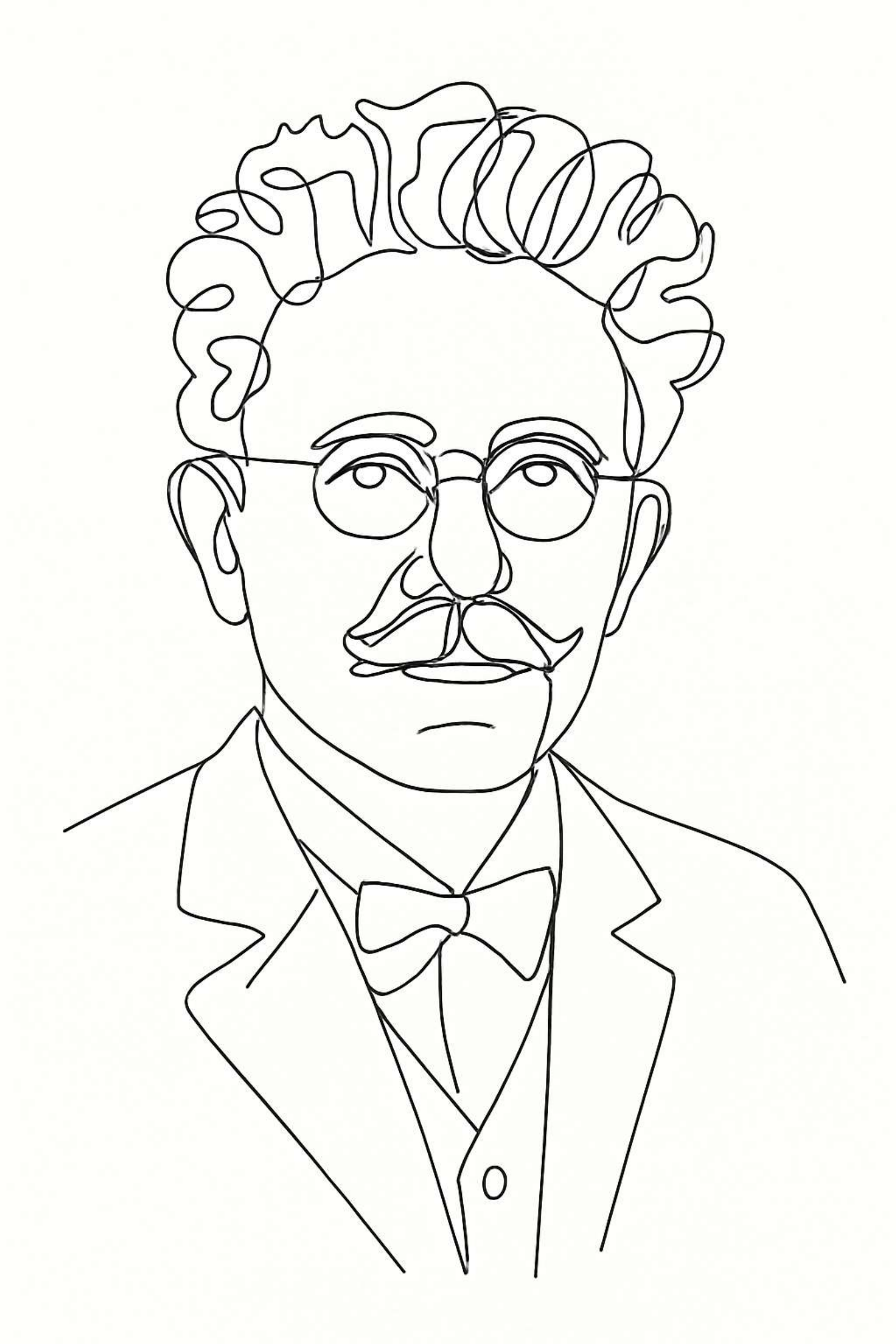}
\caption{Maurice René Fréchet (1878--1973).}
\label{figura-Frechet}
\end{figure}

The first step was taken by Maurice René Fréchet (1878--1973) (see Fig. \ref{figura-Frechet}). In his 1906 thesis \textit{Sur quelques points du calcul fonctionnel}; cf. \cite{Frechet1906}, he introduced the notions of a metric space and a functional depending on functions, establishing the conceptual framework that would extend the ideas of linear algebra to the realm of functions.

Shortly afterward, Frigyes Riesz (1880--1956) (see Fig. \ref{figura-Riesz}) gave this intuition its most precise mathematical form. In 1907, in his celebrated paper \textit{Sur les opérations fonctionnelles linéaires}; cf. \cite{Riesz1907}, he proved that every continuous linear functional on a Hilbert space $\mathcal{H}$ can be represented uniquely as an inner product with a vector from the same space. This result, now known as the Fréchet--Riesz theorem, preserved the geometric intuition of the finite-dimensional case: applying a functional is equivalent to projecting onto a direction in the space.

Thus was born the representation principle ---the idea that every linear operation can be ``embodied'' in a geometric object.

\begin{figure}[!ht]
\centering
\includegraphics[scale=0.125]{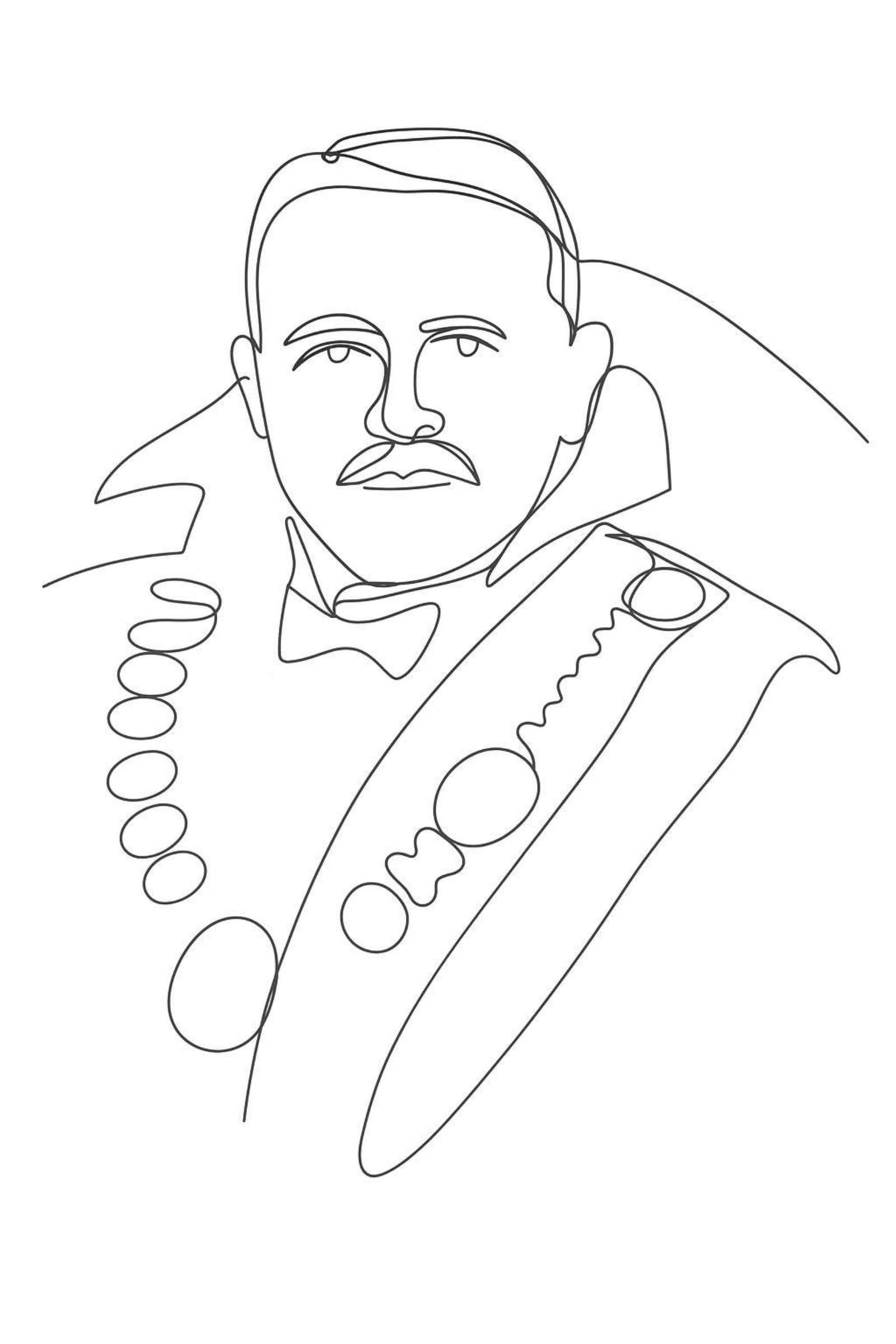}
\caption{Frigyes Riesz (1880--1956).}
\label{figura-Riesz}
\end{figure}

Riesz continued to extend his result. In 1909, in his long paper \textit{Sur les opérations fonctionnelles linéaires}; cf. \cite{Riesz1909}, he generalized the principle to the space $\mathcal{C}([a,b])$ of continuous functions on a compact interval, showing that every continuous linear functional can be represented by a Riemann--Stieltjes integral with respect to a function of bounded variation.

This Riesz--Stieltjes theorem built the first bridge between classical integral calculus and the emerging theory of measure: the coordinate vector was replaced by a function of bounded variation, and the finite sum by an integral.

Between 1910 and 1916, Riesz went a step further and formulated the duality of the $L^p$ spaces. In his 1916 paper \textit{Über lineare Funktionalgleichungen}; cf. \cite{Riesz1916}, he showed that every continuous linear functional on $L^p(a,b)$ can be represented as an integral against a function in $L^q(a,b)$, with $\frac{1}{p}+\frac{1}{q}= 1$. This celebrated $L^p$--$L^q$ duality became one of the pillars of modern analysis, showing that a linear functional can be interpreted as a \textit{generalized projection} mediated by an integrable function instead of a vector.

By the mid-1930s, the representation principle had reached full maturity. In 1937, Riesz published \textit{Sur les fonctions continues}; cf. \cite{Riesz1937}, proving that every positive and continuous linear functional on $\mathcal{C}(K)$, where $K$ is compact, corresponds uniquely to a regular Borel measure.

This gave a tangible meaning to abstract functionals: linear operations on spaces of functions could be understood as measures on their domain. The geometry of projections had evolved into the geometry of measures.

The next step came from the Russian school of analysis and probability. In 1938, Andrei Andreevich Markov Jr. (1903--1979) (see Fig. \ref{figura-Markov}) published \textit{On mean values and exterior densities}; cf. \cite{Markov1938}, in which he removed the assumption of compactness by extending Riesz's result to locally compact spaces. He proved that continuous and positive linear functionals on $\mathcal{C}_c(\Omega)$ ---the space of continuous functions with compact support--- can also be represented by regular Borel measures.

\begin{figure}[!ht]
\centering
\includegraphics[scale=0.125]{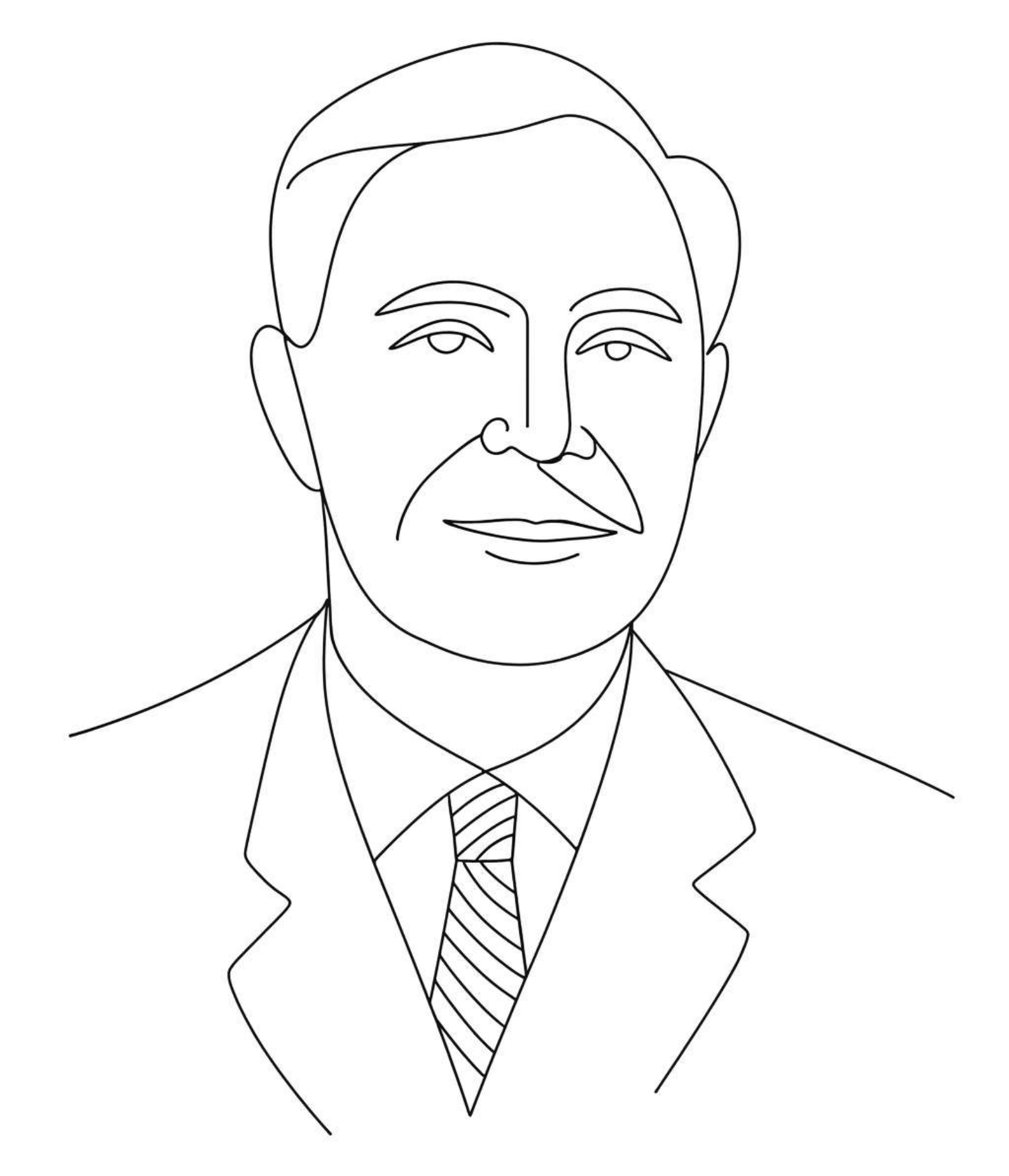}
\caption{Andrei Andreevich Markov Jr. (1903--1979).}
\label{figura-Markov}
\end{figure}

Although his work, written in Russian, had little diffusion in the West, it established a crucial link between Riesz's theory and the modern formulation later completed by Shizuo Kakutani.

Finally, in 1941, Shizuo Kakutani (1911--2004) (see Fig. \ref{figura-Kakutani}) unified the ideas of Riesz and Markov within the framework of spaces of continuous functions. In his paper \textit{Concrete representations of abstract (M)-spaces (A characterization of the space of continuous functions)}; cf. \cite{Kakutani1941}, he formulated the definitive version of the Riesz--Markov--Kakutani theorem, stating that every continuous and positive linear functional on $\mathcal{C}(\Omega)$, with $\Omega$ compact Hausdorff, can be represented by a regular Borel measure.

\begin{figure}[!ht]
\centering
\includegraphics[scale=0.125]{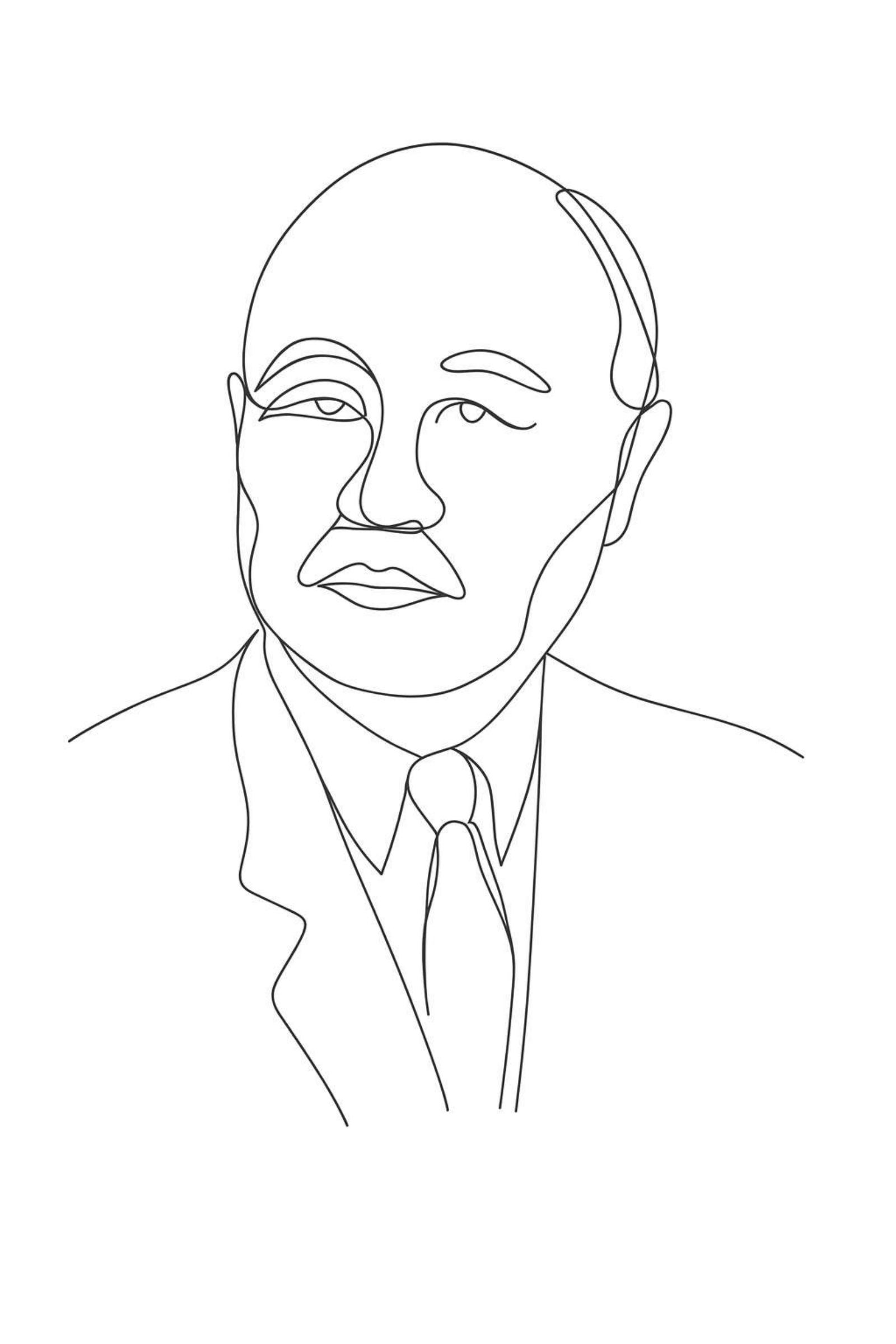}
\caption{Shizuo Kakutani (1911--2004).}
\label{figura-Kakutani}
\end{figure}

With this, the abstraction of functional analysis reached its most concrete expression: functionals became measures, and linear operations became integrals. The story that began with matrix representations in $\mathbb{R}^N$ culminated in a theory encompassing infinite-dimensional and topological structures.

The journey from Fréchet and Riesz to Kakutani not only describes a sequence of mathematical results but also narrates a profound transformation in the very way linearity is conceived. We now proceed to present each of the named representation theorems in a rigorous manner. 

\section{From Hilbert Spaces to Measures: Formulations of the Representation Principle}

This section presents the exact versions of the Riesz representation theorems that will be used throughout the rest of this work. Since each result admits different formulations depending on the functional context in which it is stated, it is convenient to fix from the outset the precise hypotheses and conclusions that we will adopt.

The order of exposition follows the natural evolution of the representation principle: we begin with the Fréchet--Riesz theorem in Hilbert spaces (see \cite[Theorem 5.25]{Folland1999}), continue with the Riesz--Stieltjes theorem on the real line (see \cite[Theorem 9.9]{Bartle1995}) and the Riesz theorem in $L^p$ spaces (see \cite[Theorem 8.15]{Bartle1995}), and conclude with the Riesz--Markov--Kakutani theorem (see \cite[Theorem 1.1]{Espejo2023}) for spaces of continuous functions with compact support.

These formulations establish the functional framework that will serve as the common foundation for the developments that follow.

\subsection{From Projection to Functional: The Fréchet--Riesz Theorem}

The starting point of the representation principle lies in Hilbert spaces, where the notion of duality is expressed through the inner product.

\begin{theorem}[Fréchet--Riesz, 1907]\label{TeoFrechetRiesz}
Let $\mathcal{H}$ be a real Hilbert space. If $L:\mathcal{H} \to \mathbb{R}$ is a linear and continuous functional, then there exists a unique vector $w \in \mathcal{H}$ such that
$$
L(u) = \langle u, w \rangle
\qquad \forall\, u \in \mathcal{H}.
$$
\end{theorem}

This result, which identifies the dual of a Hilbert space with the space itself, preserves entirely the geometric intuition of the finite-dimensional case: applying a functional is equivalent to projecting onto a specific direction in the space.

\subsection{From Continuous Functions to Integrals: The Riesz--Stieltjes Theorem}

When passing from vector spaces to spaces of functions, the notion of projection is transformed into integration.  

Consider the space $\mathcal{C}([a,b])$ of real continuous functions on a compact interval $[a,b] \subset \mathbb{R}$, endowed with the supremum norm
$$
\|f\|_{\infty} = \sup_{x \in [a,b]} |f(x)|.
$$

With this norm, $\mathcal{C}([a,b])$ is a Banach space.

\begin{theorem}[Riesz--Stieltjes, 1909]\label{RScompacto}
Let $L:\mathcal{C}([a,b]) \to \mathbb{R}$ be a linear and continuous functional.  
Then there exists a function $\alpha:[a,b] \to \mathbb{R}$ of bounded variation such that
$$
L(f) = \int_{[a,b]} f(t)\, d\lambda_{\alpha}(t)
\qquad \forall\, f \in \mathcal{C}([a,b]),
$$
where $\lambda_{\alpha}$ denotes the Lebesgue--Stieltjes measure generated by $\alpha$.  

The function $\alpha$ is unique once its value at one point is fixed and is required to be right-continuous.  
If, in addition, $L$ is positive ($L(f) \ge 0$ for every $f \ge 0$), then $\alpha$ is nondecreasing.
\end{theorem}

Riesz’s idea was to replace the vector that “represents” $L$ with a function of bounded variation, and the finite sum with an integral.  
In this way, the correspondence between functionals and measures was made explicit on the real line.

Similarly, one can formulate an analogous result for continuous functions with compact support on the real line.  
Let $\mathcal{C}_c(\mathbb{R})$ denote the space of real continuous functions with compact support (that is, functions that vanish outside a compact subset of $\mathbb{R}$), also endowed with the supremum norm.  
The result can be stated as follows.

\begin{theorem}[Riesz--Stieltjes on $\mathbb{R}$]\label{TeoRSenR}
Let $L:\mathcal{C}_c(\mathbb{R}) \to \mathbb{R}$ be a linear and continuous functional.  
Then there exists a function $\alpha:\mathbb{R} \to \mathbb{R}$ of bounded variation on every compact interval such that
$$
L(f) = \int_{\mathbb{R}} f(t)\, d\lambda_{\alpha}(t)
\qquad \forall\, f \in \mathcal{C}_c(\mathbb{R}),
$$
where $\lambda_{\alpha}$ is the associated Lebesgue--Stieltjes measure.  

The function $\alpha$ is unique (up to an initial value) and can be chosen to be right-continuous.  
If $L$ is positive, then $\alpha$ is nondecreasing.
\end{theorem}

The proof follows by extending the compact case to the entire real line. Consider the increasing family of intervals $[-n,n]$, $n \in \mathbb{N}$, and define on each of them the restricted functional $L_n = L|_{\mathcal{C}([-n,n])}$.  

Each $L_n$ is linear and continuous, and by Theorem~\ref{RScompacto} there exists a function of bounded variation $\alpha_n:[-n,n] \to \mathbb{R}$ such that
$$
L(f) = \int_{[-n,n]} f(t)\, d\lambda_{\alpha_n}(t)
\qquad \forall\, f \in \mathcal{C}([-n,n]).
$$

Fixing a normalization, for instance $\alpha_n(0)=0$, and requiring each $\alpha_n$ to be right-continuous, the uniqueness part of the theorem ensures that the functions agree on their overlaps: if $m>n$, then $\alpha_m|_{[-n,n]} = \alpha_n$.  

Hence one can define a function $\alpha:\mathbb{R} \to \mathbb{R}$ such that $\alpha(x) = \alpha_n(x)$ whenever $x \in [-n,n]$. This function has bounded variation on every compact interval and satisfies
$$
L(f) = \int_{\mathbb{R}} f(t)\, d\lambda_\alpha(t)
\qquad \forall\, f \in \mathcal{C}_c(\mathbb{R}).
$$

If, in addition, $L$ is positive, each $\alpha_n$ is nondecreasing, and by compatibility, so is $\alpha$.  
Thus, the integral representation of continuous linear functionals extends naturally to the whole real line, removing the need for compactness.

\subsection{From Continuous Functions to Lebesgue Spaces: Riesz in \texorpdfstring{$L^{p}$}{}}

The next step in the evolution of the representation principle occurs in the context of the Lebesgue spaces $L^p$, where linearity and integration intertwine in their most general form.

Let $(\Omega, \mathcal{F}, \mu)$ be a measure space and $p,q \in (1,\infty)$ be conjugate exponents, that is, such that $\frac{1}{p} + \frac{1}{q} = 1$.

\begin{theorem}[Riesz in $L^p$, 1916]\label{TeoRieszLp}
Let $L:L^p(\Omega, \mathcal{F}, \mu) \to \mathbb{R}$ be a linear and continuous functional. Then there exists a unique function $g \in L^q(\Omega, \mathcal{F}, \mu)$ such that
$$
L(f) = \int_{\Omega} f(\omega)\, g(\omega)\, d\mu(\omega)
\qquad \forall\, f \in L^p(\Omega, \mathcal{F}, \mu).
$$
\end{theorem}

This result characterizes the dual of $L^p$ as the conjugate space $L^q$, and can be viewed as a generalization of the Fréchet--Riesz theorem from inner products to integral products.

\subsection{The Principle in Its General Form: The Riesz--Markov--Kakutani Theorem}

The culmination of the representation principle is achieved in the general topological setting.  

Let $(\Omega, \tau)$ be a locally compact Hausdorff space, and let $\mathcal{C}_c(\Omega)$ denote the space of continuous functions with compact support, endowed with the supremum norm.

\begin{theorem}[Riesz--Markov--Kakutani, 1941]\label{TeoRieszMarkovKakutani}
Let $L:\mathcal{C}_c(\Omega) \to \mathbb{R}$ be a linear, positive, and continuous functional.  
Then there exists a unique regular Borel measure $\mu$ on $\Omega$ such that
$$
L(f) = \int_{\Omega} f(\omega)\, d\mu(\omega)
\qquad \forall\, f \in \mathcal{C}_c(\Omega).
$$
\end{theorem}

This theorem, arising from the synthesis of results by Riesz (1909, 1937), Markov (1938), and Kakutani (1941), provides the most general formulation of the representation principle: continuous and positive linear functionals on spaces of continuous functions correspond exactly to regular Borel measures.

With these statements, we have established the functional and notational framework in which we shall work.  
Each of them illustrates a stage in the same geometric principle that inspired Riesz: every continuous linear action can ultimately be understood as a form of integration. 

The following and last section of this work presents how fundamental concepts of probability theory (such as expectation, distribution, conditional expectation, and the Wiener measure) are analytic manifestations of a single principle of representation. Furthermore, fundamental probability theory appears not merely as an extension of measure theory but as the geometric realization of functional analysis itself.

\section{The Representation Principle in Probability Theory}

\subsection{The Fréchet--Riesz Theorem: Expectation as a Vector}

In classical probability theory, it is common to study random variables taking values in $\mathbb{R}^N$, where concepts such as expectation, variance, and distribution admit well-known geometric and analytic interpretations. However, in more general settings ---for instance, in the study of stochastic partial differential equations, randomly evolving processes, or the analysis of random fields--- it is natural to consider random variables that take values in infinite-dimensional Banach or Hilbert spaces; cf. \cite{DaPrato1999}. In such situations, the notion of ``expected value'' is no longer immediate: integrating coordinates is not enough; one needs a notion of integral that is consistent with the vector and topological structure of the space.

The Bochner integral arises precisely to meet this need. It is the natural generalization of the Lebesgue integral to the context of functions with values in Banach spaces, and its goal is to extend the classical concept of expectation to the case of vector-valued random variables. In this framework, integrability is formulated in terms of the space norm and strong measurability, and the integral is defined as the element of the space that coherently reproduces the action of all continuous linear functionals on the integrable function (see \cite[Chs. 2 and 3]{Dinov1993}). This construction preserves the essential properties of scalar integration ---linearity, continuity, and uniqueness--- and provides a solid foundation for the study of random variables in abstract vector spaces.

When the Banach space is in fact a separable Hilbert space, the theory becomes notably transparent. Thanks to the Fréchet--Riesz representation theorem, which identifies each continuous linear functional with a unique vector of the Hilbert space itself, the existence and uniqueness of the expectation of a vector-valued random variable can be derived directly. Thus, the Bochner integral not only offers an elegant and rigorous construction of the expected value in this context, but also highlights the deep connection between the principles of functional analysis and the foundations of modern probability: both revolve around the same idea of linear representation and structural duality.

Let $(\Omega,\mathcal{F},\mathbb{P})$ be a complete probability space and let $(\mathcal{H},\langle \cdot,\cdot\rangle,\|\cdot\|)$ be a separable Hilbert space over $\mathbb{R}$.

The topological structure of $\mathcal{H}$ naturally induces its Borel $\sigma$-algebra, denoted $\mathcal{B}(\mathcal{H})$, generated by the open sets of the space. The pair $(\mathcal{H},\mathcal{B}(\mathcal{H}))$ is therefore a measurable space. In this setting, a function $X:\Omega\to\mathcal{H}$ is called an $\mathcal{H}$-valued random variable if it is measurable, i.e., if for every open set $\mathcal{O}\subset \mathcal{H}$ we have $X^{-1}(\mathcal{O})\in\mathcal{F}$.

Every random variable $X:\Omega\to\mathcal{H}$ naturally induces a probability measure on $\mathcal{H}$ describing how the values of $X$ are distributed in the Hilbert space. This measure, called the distribution of $X$, is the pushforward $X_{\sharp}\mathbb{P}:\mathcal{B}(\mathcal{H})\to[0,1]$ defined by
$$
X_{\sharp}\mathbb{P}(\mathcal{O}) := \mathbb{P}\!\big(X^{-1}(\mathcal{O})\big).
$$

The measure $X_{\sharp}\mathbb{P}$ is a probability measure on $(\mathcal{H},\mathcal{B}(\mathcal{H}))$ and satisfies the change-of-variables formula, namely,
$$
\int_{\Omega} f\big(X(\omega)\big)\, d\mathbb{P}(\omega)
= \int_{\mathcal{H}} f(v)\, d\big(X_{\sharp}\mathbb{P}\big)(v),
$$
for every measurable $f:\mathcal{H}\to\mathbb{R}$ (see \cite[Theorem 2.1]{DaPrato1999}). This identity allows one to formulate integrals directly on $\mathcal{H}$ by transporting $\mathbb{P}$ to the value space of the random variable $X$.

Let $X:\Omega\to\mathcal{H}$ be a random variable. To analyze the size of the values taken by $X$, consider the norm map $\|\cdot\|:\mathcal{H}\to\mathbb{R}$. Since the norm is continuous, it is measurable; consequently, the composition $\|\cdot\|\circ X:\Omega\to\mathbb{R}$ defines a real random variable. This permits one to quantify the magnitude of the values of $X$ and to formulate integrability conditions in $\mathcal{H}$. In particular, we shall be interested in those random variables that satisfy
$$
\int_{\mathcal{H}} \|v\|\, d\big(X_{\sharp}\mathbb{P}\big)(v) < \infty,
$$
a condition that guarantees finiteness of the expected norm of $X$.

Under this hypothesis it will be possible to define a linear and continuous functional on $\mathcal{H}$ by integrating with respect to the distribution of $X$.

Let $X:\Omega\to\mathcal{H}$ be a random variable such that
$$
\int_{\mathcal{H}} \|v\|\, d\big(X_{\sharp}\mathbb{P}\big)(v) < \infty.
$$

Define $L:\mathcal{H}\to\mathbb{R}$ by
$$
L(u):=\int_{\mathcal{H}} \langle u,v\rangle \, d\big(X_{\sharp}\mathbb{P}\big)(v),
$$
which is clearly linear by the properties of the inner product on $\mathcal{H}$. Moreover, applying the Cauchy--Schwarz inequality we obtain
$$
|L(u)| \le \int_{\mathcal{H}} |\langle u,v\rangle|\, d\big(X_{\sharp}\mathbb{P}\big)(v)\, d\big(X_{\sharp}\mathbb{P}\big)(v)
\le \|u\| \int_{\mathcal{H}} \|v\| \, d\big(X_{\sharp}\mathbb{P}\big)(v)
\quad \forall\, u\in\mathcal{H},
$$
which implies that $L$ is continuous on $\mathcal{H}$.

By the Fréchet--Riesz representation theorem (see Theorem \ref{TeoFrechetRiesz}), there exists a unique $\mu\in\mathcal{H}$ such that
$$
\int_{\mathcal{H}} \langle u,v\rangle\, d\big(X_{\sharp}\mathbb{P}\big)(v) \;=\; L(u) \;=\; \langle u,\mu\rangle
\qquad \forall\, u\in\mathcal{H}.
$$

This vector $\mu$ characterizes the functional $L$ and, consequently, the random variable $X$ through its distribution. The element $\mu$ is called the \textbf{Bochner integral} of $X$ and represents its \textbf{expectation}; for convenience, we write $\mu=\mathbb{E}(X)$.

We now present a detailed example.

\begin{example}
Consider the probability space $((0,1),\mathcal{L}(0,1),\lambda)$, where $\mathcal{L}(0,1)$ denotes the Lebesgue $\sigma$-algebra and $\lambda$ the Lebesgue measure. Let $L^2(0,1)$ be the separable Hilbert space endowed with the inner product (see \cite[Chapter 23]{JacodProtter2004})
$$
\langle u,v\rangle \;=\; \int_{(0,1)} u(t)\,v(t)\, d\lambda(t),
\qquad u,v \in L^2(0,1).
$$

Define $\chi:(0,1)\to L^{2}(0,1)$ by
$$
\chi(\omega):=\mathbf{1}_{(0,\omega)}.
$$

We claim that $\chi$ is an $L^2(0,1)$-valued random variable. Indeed, the continuity of $\chi$ with respect to the $L^2(0,1)$ norm can be verified directly: given $\omega_0\in(0,1)$ and $\varepsilon>0$, letting $\delta:=\varepsilon^2>0$, for every $\omega\in(0,1)$ with $|\omega_{0}-\omega|<\delta$ we have
$$
\begin{aligned}
\|\chi(\omega_0)-\chi(\omega)\|_{L^2(0,1)}^2
&= \int_{(0,1)} \big|\mathbf{1}_{(0,\omega_0)}(t)-\mathbf{1}_{(0,\omega)}(t)\big|^2\, d\lambda(t) \\
&= \int_{(0,1)} \mathbf{1}_{(0,\omega_0)\triangle(0,\omega)}(t)\, d\lambda(t) \\
&= \lambda\big((0,\omega_0)\triangle(0,\omega)\big)= |\omega_0-\omega| < \varepsilon^2,
\end{aligned}
$$
where $(0,\omega_0)\triangle(0,\omega)$ denotes the symmetric difference. Hence $\chi$ is continuous, and therefore measurable.

For each $\omega\in(0,1)$ we also have
$$
\|\chi(\omega)\|_{L^2(0,1)}^{2}
= \int_{(0,1)} \big|\mathbf{1}_{(0,\omega)}(t)\big|^2\, d\lambda(t)
= \lambda(0,\omega)=\omega.
$$

Using the change-of-variables formula, we obtain
$$
\begin{aligned}
\int_{L^2(0,1)} \|v\|_{L^2(0,1)}\, d\big(\chi_{\sharp}\lambda\big)(v)
&= \int_{(0,1)} \|\chi(\omega)\|_{L^2(0,1)} \, d\lambda(\omega) \\
&= \int_{(0,1)} \left( \int_{(0,1)} \big|\mathbf{1}_{(0,\omega)}(t)\big|^2\, d\lambda(t) \right)^{\!1/2} d\lambda(\omega) \\
&= \int_{(0,1)} \omega^{1/2}\, d\lambda(\omega)
= \frac{2}{3}.
\end{aligned}
$$

Let $u\in L^{2}(0,1)$ be arbitrary. Again, by the change-of-variables formula and Fubini’s theorem (see \cite[Theorem 1.7.2]{Durrett2019}),
$$
\begin{aligned}
\int_{L^2(0,1)} \langle v,u\rangle\, d\big(\chi_{\sharp}\lambda\big)(v)
&= \int_{(0,1)} \langle \chi(\omega),u \rangle\, d\lambda(\omega)\\
&= \int_{(0,1)} \left( \int_{(0,1)} \mathbf{1}_{(0,\omega)}(t)\,u(t)\, d\lambda(t) \right) d\lambda(\omega)\\
&= \int_{(0,1)} \left( \int_{(0,1)} \mathbf{1}_{(t,1)}(\omega)\, u(t)\, d\lambda(\omega) \right) d\lambda(t)\\
&= \int_{(0,1)} (1-t)\,u(t)\, d\lambda(t).
\end{aligned}
$$

By the Fréchet--Riesz theorem, the expectation (Bochner integral) of the random variable $\chi$ is the function $\mu:(0,1)\to\mathbb{R}$ given by $\mu(t)=1-t$, which belongs to $L^2(0,1)$ (see Fig. \ref{Bochner-integral}).

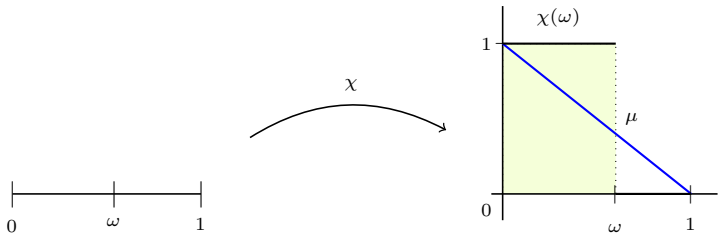
\begin{figure}[!ht]
    \centering
	\begin{tikzpicture}[xscale=1,yscale=1]
	\draw[-, semithick] (0,0)--(2.5,0);
	\draw (0,0) node{$|$}; \draw (0,-0.22) node[below]{$_{0}$};
	\draw (2.5,0) node{$|$}; \draw (2.5,-0.2) node[below]{$_{1}$};
	\draw (1.35,0) node{$|$}; \draw (1.35,-0.2) node[below]{$_{\omega}$};

    \draw[<-, semithick] (5.75,0.85) to[bend right] (3.15,0.75);
    \draw(4.5,1.25)node[above]{$_{\chi}$};
    \draw [fill=lime!15,  dotted]  (6.5,0)--(6.5,2)--(8,2)--(8,0);
	\draw[-,semithick] (6.35,0) to (9.35,0);
	\draw[-, semithick] (6.5,-0.35)--(6.5,2.5);
    \draw[-, thick, blue] (6.5,2)--(9,0);
    \draw (8,1) node[right]{$_{\mu}$};
    \draw (7.25,2.1) node[above]{$_{\chi(\omega)}$};
    \draw (6.5,2) node{$_{-}$}; 
    \draw (6.5,2) node[left] {$_{1}$}; 
    \draw (6.5,-0.22) node[left] {$_{0}$};
    \draw (8,0) node{$_{|}$}; 
    \draw (8,-0.27) node[below] {$_{\omega}$};
    \draw (9,0) node{$_{|}$}; 
    \draw (9,-0.2) node[below] {$_{1}$};
    \draw [thick,black] (8,0)--(9,0);
    \draw [thick,black] (6.5,2)--(8,2);
\end{tikzpicture}
\caption{Graphical representation of the expected value of the random variable $\chi:(0,1)\to L^{2}(0,1)$.}
\label{Bochner-integral}
\end{figure}
\end{example}

Thus, the notion of expectation in separable Hilbert spaces arises naturally as a consequence of the Fréchet--Riesz theorem: the expected value of an $\mathcal{H}$-valued random variable is the unique vector of $\mathcal{H}$ that represents the integration functional induced by $X$.

\subsection{The Riesz--Stieltjes Theorem: Distribution as representative measure}

In modern probability theory, the concept of the \textit{distribution function} constitutes one of the fundamental pillars for describing the behavior of a real random variable $X$. The function
$$
F_X(x) = \mathbb{P}\big(X^{-1}((-\infty, x])\big),
$$
encodes all the probabilistic information about $X$: its law, its moments, the possibility of computing expected values, and, more generally, its statistical structure. The distribution function allows one to translate random phenomena from the abstract space $(\Omega, \mathcal{F}, \mathbb{P})$ into the real domain, where measure theory and integration provide a precise framework for their analysis (see, for example, \cite[\textsection 1.2]{Durrett2019}).

Beyond its elementary definition, the existence and properties of $F_X$ can be understood as a profound consequence of a functional--analytic principle: the Riesz--Stieltjes representation theorem. This result establishes a correspondence between positive linear functionals on the space $\mathcal{C}_c(\mathbb{R})$ and finite Borel measures on $\mathbb{R}$, which, in turn, can be represented by right--continuous functions of bounded variation.

From this perspective, the distribution function emerges not merely as a probabilistic object, but as the concrete manifestation of the correspondence between functional analysis and measure theory: every expected value is the evaluation of a linear functional which, by virtue of the Riesz--Stieltjes theorem, admits an integral representation with respect to a bounded variation function. In the probabilistic context, that function is precisely the distribution function.

Let $(\Omega, \mathcal{F}, \mathbb{P})$ be a probability space and let $X:\Omega \to \mathbb{R}$ be a real random variable. Consider the functional $L:  \mathcal{C}_{c}(\mathbb{R}) \to \mathbb{R}$ defined by
$$
L(f) := \mathbb{E}\big(f(X)\big) = \int_{\Omega} f(X(\omega))\, d\mathbb{P}(\omega).
$$

By the properties of the Lebesgue integral, $L$ is linear and positive. Moreover,
$$
|L(f)| \le \int_{\Omega} |f(X(\omega))|\, d\mathbb{P}(\omega) 
\leq \Vert f \Vert_{\infty} \int_{\Omega} 1\, d\mathbb{P}(\omega)
= \Vert f \Vert_{\infty}, \qquad \forall f\in \mathcal{C}_{c}(\mathbb{R}),
$$
which implies that $L$ is continuous on $\mathcal{C}_{c}(\mathbb{R})$. 

Applying the Riesz--Stieltjes representation theorem to $L$ (see Theorem \ref{TeoRSenR}), we obtain a right-continuous nondecreasing function $\alpha:\mathbb{R}\to\mathbb{R}$ such that
\begin{equation}\label{eq:RS-repr}
L(f) = \int_{\mathbb{R}} f(t)\,d\lambda_\alpha(t), \qquad \forall f\in \mathcal{C}_c(\mathbb{R}),
\end{equation}
where $\lambda_{\alpha}$ denotes its associated Lebesgue--Stieltjes measure.

For each $j \in \mathbb{N}$, consider a continuous function $\psi_j:\mathbb{R} \to \mathbb{R}$ satisfying $0\le \psi_j\le 1$, $\psi_j=1$ on $[-j,j]$, and $\psi_j=0$ on $\mathbb{R}\smallsetminus [-(j+1),j+1]$ (see Fig. \ref{Graph--psi}).

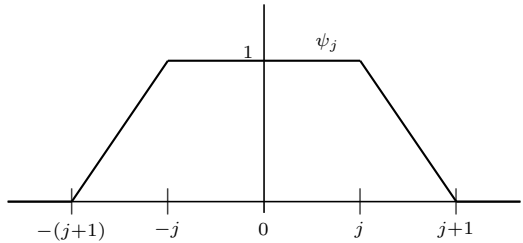
\begin{figure}[!ht]
    \centering
    \begin{tikzpicture}[xscale=0.85,yscale=0.75]
	\draw[-, semithick] (-4,0)--(4,0);
    \draw[-, semithick] (0,-0.2)--(0,3.5);
	\draw (0,-0.22) node[below]{$_{0}$};
	\draw (1.5,0) node{$|$}; \draw (1.5,-0.2) node[below]{$_{j}$};
    \draw (-1.5,0) node{$|$}; \draw (-1.5,-0.2) node[below]{$_{-j}$};
	\draw (3,0) node{$|$}; \draw (3,-0.2) node[below]{$_{j+1}$};
    \draw (-3,0) node{$|$}; \draw (-3,-0.2) node[below]{$_{-(j+1)}$};

    \draw[-, thick] (-1.5,2.5)--(1.5,2.5);
    \draw[-, thick] (-1.5,2.5)--(-3,0);
    \draw[-, thick] (1.5,2.5)--(3,0);
    \draw[-, thick] (-3,0)--(-4,0);
    \draw[-, thick] (3,0)--(4,0);

    \draw (1,2.5) node[above]{$_{\psi_{j}}$};
    \draw (0,2.65) node[below,left]{$_{1}$};
\end{tikzpicture}
\caption{Graph of the function $\psi_{j}$}
\label{Graph--psi}
\end{figure}

Hence $(\psi_j)$ is a nondecreasing sequence of elements in $\mathcal{C}_{c}(\mathbb{R})$ which converges pointwise to the constant function $1$ on $\mathbb{R}$. By the monotone convergence theorem (see \cite[Theorem 4.6]{Bartle1995}), we obtain
$$
\begin{aligned}
\lambda_\alpha(\mathbb{R})&= \int_{\mathbb{R}} 1(t)\,d\lambda_\alpha(t)\\
&= \lim_{j\to\infty} \int_{\mathbb{R}} \psi_j(t)\,d\lambda_\alpha(t)\\
&= \lim_{j\to\infty} L(\psi_j)\\
&= \lim_{j\to\infty} \mathbb{E}[\psi_j(X)]\\
&= \mathbb{E}[1] = 1,
\end{aligned}
$$
and therefore $\lambda_\alpha$ is a probability measure.

Let $x\in\mathbb{R}$ be fixed but arbitrary. For each $j\in\mathbb{N}$, define the continuous function $f_j:\mathbb{R} \to \mathbb{R}$ by
$$
f_j(t)=
\begin{cases}
1, & t\le x,\\[3pt]
1-j(t-x), & x\le t\le x+\frac{1}{j},\\[3pt]
0, & t\ge x+\frac{1}{j}.
\end{cases}
$$

Hence $(f_j)$ is a nondecreasing sequence of continuous functions on $\mathbb{R}$ such that $f_j\to \mathbf{1}_{(-\infty,x]}$ pointwise on $\mathbb{R}$ (see Fig. \ref{Graph--fj}).

\begin{figure}[!ht]
    \centering
    \begin{tikzpicture}[xscale=0.85,yscale=0.75]
	\draw[-, semithick] (-4,0)--(4,0);
    \draw[-, semithick] (0,-0.2)--(0,3.5);
	
    \draw (0,-0.22) node[below]{$_{0}$};
	
    \draw (1,0) node{$|$}; \draw (1,-0.2) node[below]{$_{x}$};
    
	\draw (2.5,0) node{$|$}; \draw (2.5,-0.2) node[below]{$_{x+\frac{1}{j}}$};
    
    \draw[-, thick] (-4,2.5)--(1,2.5);
    
    \draw[-, thick] (1,2.5)--(2.5,0);
    \draw[-, thick] (2.5,0)--(4,0);

    \draw (1,2.5) node[above]{$_{f_{j}}$};
    \draw (0,2.65) node[below,left]{$_{1}$};
\end{tikzpicture}
\caption{Graph of the function $f_{j}$}
\label{Graph--fj}
\end{figure}
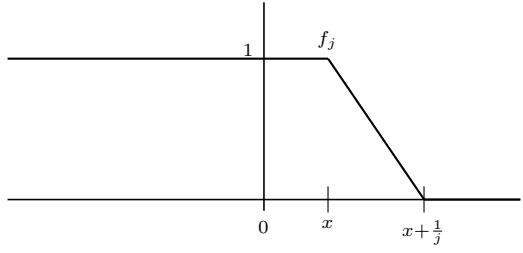

Since $f_j\notin \mathcal{C}_{c}(\mathbb{R})$ for $j \in \mathbb{N}$, we consider a truncation of them. For each $m \in \mathbb{N}$, choose arbitrary $\theta_m\in \mathcal{C}_c(\mathbb{R})$ such that $0\leq \theta_m\leq 1$, $\theta_m = 1$ on $[-m,m]$, and $\theta_m = 0$ on $\mathbb{R}\smallsetminus [-(m+1),m+1]$. Define the function $g_{j,m}:\mathbb{R} \to \mathbb{R}$ by $g_{j,m}:=f_j\,\theta_m$.  
Hence $g_{j,m} \in \mathcal{C}_{c}(\mathbb{R})$ for all $j,m \in \mathbb{N}$ (see for example Fig. \ref{Truncation--fj}).

\begin{figure}[!ht]
    \centering
    \begin{tikzpicture}[xscale=1,yscale=0.75]
	\draw[-, semithick] (-4,0)--(4,0);
    \draw[-, semithick] (0,-0.2)--(0,3.5);
	\draw (0,-0.22) node[below]{$_{0}$};
	\draw (1.5,0) node{$|$}; \draw (1.5,-0.2) node[below]{$_{m}$};
    \draw (-1.5,0) node{$|$}; \draw (-1.5,-0.2) node[below]{$_m$};
	\draw (3,0) node{$|$}; \draw (3,-0.2) node[below]{$_{m+1}$};
    \draw (-3,0) node{$|$}; \draw (-3,-0.2) node[below]{$_{-(m+1)}$};

    \draw[-, thick] (-1.5,2.5)--(1.5,2.5);
    \draw[-, thick] (-1.5,2.5)--(-3,0);
    \draw[-, thick] (1.5,2.5)--(3,0);
    \draw[-, thick] (-3,0)--(-4,0);
    \draw[-, thick] (3,0)--(4,0);

    \draw (1,2.5) node[above]{$_{\theta_{m}}$};
    \draw (0,2.65) node[below,left]{$_{1}$};
     \draw (0.75,-0.2) node[below]{$_{x}$};
     \draw (0.75,0) node{$_{|}$};

     \draw (1.125,-0.2) node[below]{$_{x+\frac{1}{j}}$};
     \draw (1.125,0) node{$_{|}$};

     \draw[-, thick, blue] (-4,2.5)--(0.75,2.5);
     \draw[-, thick, blue] (0.75,2.5)--(1.125,0);
     \draw[-, thick, blue] (1.125,0)--(4,0);
\end{tikzpicture}
\caption{Truncation of $f_{j}$}
\label{Truncation--fj}
\end{figure}
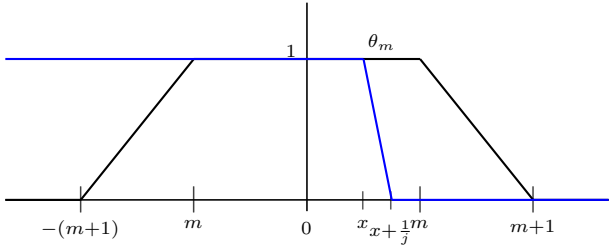

For fixed $j \in \mathbb{N}$, we have $g_{j,m}\to f_j$ pointwise on $\mathbb{R}$ as $m\to\infty$. Applying \eqref{eq:RS-repr} to $g_{j,m}$ and taking the limit as $m\to\infty$, by the monotone convergence theorem we obtain
$$
\mathbb{E}(f_j(X)) = \int_{\mathbb{R}} f_j(t)\, d\lambda_\alpha(t).
$$

Now, taking the limit as $j\to\infty$ and using the monotone convergence theorem once again, we get
$$
\mathbb{E}\big(\mathbf{1}_{X^{-1}(-\infty,x]}\big)
= \int_{\mathbb{R}} \mathbf{1}_{(-\infty,x]}(t)\, d\lambda_\alpha(t).
$$

Consequently,
$$
\lambda_\alpha((-\infty,x]) = \mathbb{P}(X^{-1}((-\infty,x])).
$$

Since the family $\{(-\infty,x]:x\in\mathbb{R}\}$ is a $\pi$-system that generates $\mathcal{B}(\mathbb{R})$, Dynkin’s theorem (see \cite[Theorem A.1.5]{Durrett2019}) implies that 
\begin{equation}\label{eq:lambda-ley}
\lambda_\alpha(B) = \mathbb{P}(X^{-1}(B)), \qquad \forall B\in\mathcal{B}(\mathbb{R}).
\end{equation}

Since $\alpha:\mathbb{R} \to \mathbb{R}$ is right-continuous and nondecreasing, the one-sided limits
$$
c_- := \lim_{x\to-\infty}\alpha(x), \qquad 
c_+ := \lim_{x\to+\infty}\alpha(x),
$$
exist, and $\lambda_\alpha((-\infty,x]) = \alpha(x) - c_-$. It follows that, for any $x \in \mathbb{R}$,
$$
\alpha(x) - c_- = \mathbb{P}(X \leq x).
$$

Defining $F:\mathbb{R} \to \mathbb{R}$ by $F(x) := \alpha(x) - c_-$, we see that $F$ is right-continuous, nondecreasing, and satisfies
$$
F(x) = \mathbb{P}(X^{-1}((-\infty,x])), \qquad
\lim_{x\to-\infty}F(x) = 0, \qquad
\lim_{x\to+\infty}F(x) = \lambda_\alpha(\mathbb{R}) = 1.
$$

Finally, using \eqref{eq:RS-repr} and the fact that $\lambda_\alpha = \lambda_F$, we conclude that
$$
\mathbb{E}\big(f(X)\big) 
= \int_{\mathbb{R}} f(t)\, d\lambda_\alpha(t)
= \int_{\mathbb{R}} f(x)\, dF(x),
\qquad \forall f\in \mathcal{C}_c(\mathbb{R}).
$$

Hence $F:\mathbb{R} \to \mathbb{R}$ is the distribution function of the random variable $X$.

The Riesz--Stieltjes theorem thus shows that the distribution function of a random variable arises naturally within the framework of representation theorems.  

Identifying the expectation functional with a Lebesgue--Stieltjes integral, we see that every probability Borel measure on $\mathbb{R}$ can be described by a right-continuous, nondecreasing function whose variation reproduces exactly the cumulative probability.

\subsection{The Riesz theorem in \texorpdfstring{$L^{p}$}{}: conditional expectation}

Conditional expectation is one of the fundamental notions of modern probability, as it describes expectations when the available information is restricted to a sub-$\sigma$-algebra (see \cite[Chapter 4]{Durrett2019}). Beyond its probabilistic interpretation, this concept possesses a deeply analytic structure that can be understood within the general framework of $L^{p}$ spaces.

In this section, we show that the existence and uniqueness of conditional expectation follow directly from the Riesz representation theorem in $L^{p}$. This theorem, which identifies continuous linear functionals with integrals with respect to elements of the dual space $L^{q}$, provides the necessary tools to formulate conditional expectation as the unique function in $L^{p}$ that represents the same functional originally defined in $L^{q}$.

Thus, the construction of conditional expectation requires neither the Radon--Nikodým theorem nor geometric arguments in $L^{2}$, but arises naturally from the structural duality between $L^{p}$ and $L^{q}$; cf. \cite{ReynaSandoval2025}. This perspective reveals its genuine functional nature and situates conditionality within the same representation principle from which integration originates, highlighting the conceptual unity between the two processes.

\begin{definition}
    Let $(\Omega,\mathcal{F},\mathbb{P})$ be a probability space, $\mathcal{G}$ a sub-$\sigma$-algebra of $\mathcal{F}$ and a random variable $X\in L^{1}(\Omega,\mathcal{F},\mathbb{P})$. The \textbf{conditional expectation} of $X$ given $\mathcal{G}$ is the unique function $\mathbb{E}(X \mid \mathcal{G}) \in L^{1}(\Omega,\mathcal{G},\mathbb{P})$ such that
\begin{equation}
    \int_{\Omega} X\,\mathbf{1}_{A}\,d\mathbb{P}=\int_{\Omega} \mathbb{E}(X \mid \mathcal{G})\,\mathbf{1}_{A}\,d\mathbb{P} \qquad \forall \,A \in\mathcal{G}.
\end{equation}
\end{definition}

We begin by establishing the concept of conditional expectation for random variables in $L^{p}(\Omega,\mathcal{F},\mathbb{P})$, and then show that this result extends naturally to the case $L^{1}(\Omega,\mathcal{F},\mathbb{P})$ (see \cite[Lemma 23.1]{JacodProtter2004}).

Let $\mathcal{G}$ be a sub-$\sigma$-algebra of $\mathcal{F}$, and let $p,q\in (1,\infty)$ be such that $\frac{1}{p}+\frac{1}{q}=1$. Suppose $X \in L^{p}(\Omega,\mathcal{F},\mathbb{P})$. 

Consider the functional $L : L^{q}(\Omega,\mathcal{G},\mathbb{P}) \to \mathbb{R}$ defined by
$$
L(Y) := \int_{\Omega} X\,Y\, d\mathbb{P},
$$
which is clearly linear. By Hölder’s inequality (see \cite[Theorem 23.10]{JacodProtter2004}) we have,
$$
|L(Y)| \leq \int_{\Omega} |X\,Y|\, d\mathbb{P} 
\leq \|X\|_{L^{p}(\Omega)}\, \|Y\|_{L^{q}(\Omega)}, 
\qquad \forall\, Y \in L^{q}(\Omega,\mathcal{G},\mathbb{P}),
$$
and hence $L$ is continuous on $L^{q}(\Omega,\mathcal{G},\mathbb{P})$.

By the Riesz representation theorem in $L^{p}$ (see Theorem \ref{TeoRieszLp}), there exists a unique $\xi \in L^{p}(\Omega,\mathcal{G},\mathbb{P})$ such that
$$
L(Y) = \int_{\Omega} \xi\, Y\, d\mathbb{P},
\qquad \forall\, Y \in L^{q}(\Omega,\mathcal{G},\mathbb{P}).
$$

In particular,
\begin{equation} \label{For:LpEsperanzacondicionalIde}
\int_{\Omega} X\,\mathbf{1}_{A} \, d\mathbb{P}
= \int_{\Omega} \xi\,\mathbf{1}_{A} \, d\mathbb{P},
\qquad \forall\, A \in \mathcal{G}.
\end{equation}

Thus, we identify $\xi$ with the conditional expectation of $X$ given $\mathcal{G}$, that is, $\xi = \mathbb{E}(X \mid \mathcal{G})$ in $L^{p}(\Omega,\mathcal{G},\mathbb{P})$.

This construction preserves positivity: if $X \geq 0$, then $\mathbb{E}(X \mid \mathcal{G}) \geq 0$ almost surely on $\Omega$. Indeed, let 
$$
A(\xi) := \{\omega \in \Omega : \mathbb{E}(X \mid \mathcal{G})(\omega) < 0\} \in \mathcal{G}.
$$

Since $\mathbf{1}_{A(\xi)} \geq 0$, it follows that $X\,\mathbf{1}_{A(\xi)} \ge 0$, and by \eqref{For:LpEsperanzacondicionalIde},
$$
\mathbb{E}\!\left(\mathbb{E}(X \mid \mathcal{G})\,\mathbf{1}_{A(\xi)}\right)
= \mathbb{E}(X\,\mathbf{1}_{A(\xi)}) \geq 0.
$$

However, if $\mathbb{P}(A(\xi)) > 0$, the integrand is strictly negative on $A(\xi)$, which contradicts the previous equality. Hence $\mathbb{P}(A(\xi)) = 0$, and therefore $\mathbb{E}(X \mid \mathcal{G}) \geq 0$ almost surely on $\Omega$.

As an immediate consequence, if $X_1, X_2 \in L^{p}(\Omega,\mathcal{F},\mathbb{P})$ satisfy $X_1 \leq X_2$, then 
$$
\mathbb{E}(X_1 \mid \mathcal{G}) \leq \mathbb{E}(X_2 \mid \mathcal{G})
$$
almost surely on $\Omega$.

This establishes the result in the $L^{p}(\Omega,\mathcal{F},\mathbb{P})$ case. We now extend the construction to the general case $L^{1}(\Omega,\mathcal{F},\mathbb{P})$, showing that the same argument can be obtained as a limit of monotone approximations (see \cite[Lemma 23.1]{JacodProtter2004}).

\begin{theorem}
Let $\mathcal{G}$ be a sub-$\sigma$-algebra of $\mathcal{F}$. For any random variable $X \in L^{1}(\Omega,\mathcal{F},\mathbb{P})$ there is a unique 
$\xi \in L^{1}(\Omega,\mathcal{G},\mathbb{P})$ such that
$$
\int_{A} X\, d\mathbb{P} = \int_{A} \xi\, d\mathbb{P}, 
\qquad \forall\, A \in \mathcal{G}.
$$

In consequence, the conditional expectation $\mathbb{E}(X \mid \mathcal{G})$ exists.
\end{theorem}

\begin{proof}
Let $X \in L^{1}(\Omega,\mathcal{F},\mathbb{P})$ and suppose that $X \geq 0$ on $\Omega$.  

For each $j \in \mathbb{N}$, define the random variable $X_{j}:\Omega \to \mathbb{R}$ by
$$
X_j(\omega) := \min\{X(\omega),\, j\}.
$$

Clearly $X_j \in L^{p}(\Omega,\mathcal{F},\mathbb{P})$ since $|X_j|^p \leq j^p$, and $X_{j} \leq X_{j+1}$ for each $j \in \mathbb{N}$ (see Fig. \ref{Truncation--X}).

Thus, $(X_j)$ is a nondecreasing sequence in $L^{p}(\Omega,\mathcal{G},\mathbb{P})$ such that $X_j \to X$ pointwise on $\Omega$. Indeed, let $\omega \in \Omega$. Since $X(\omega) \geq 0$, there exists $j(\omega) \in \mathbb{N}$ such that $X(\omega) < j(\omega)$, and consequently $X_{j}(\omega) = X(\omega)$ for all $j \geq j(\omega)$.

\begin{figure}[!ht]
\centering
\begin{tikzpicture}[scale=0.85]
\draw[-, semithick] (0,0) -- (4.2,0);
\draw[-, semithick] (0,0) -- (0,3.2);

\draw[domain=0:3, smooth, variable=\x, thick] plot ({\x}, {0.3*\x*\x+0.5});
\node[right, semithick] at (2.8,2.6) {$X$};

\draw[domain=0:1.5, smooth, variable=\x, blue, thick] plot ({\x}, {0.3*\x*\x+0.5});

\draw[dashed, thick] (0,1.18) -- (1.5,1.18);
\draw[blue, thick] (1.5,1.18) -- (3.8,1.18);

\draw[dashed, semithick] (1.5,0) -- (1.5,1.18);
\node[left] at (0,1.18) {$j$};

\node[above] at (3.5,1.15) {$X_{j}$};

\end{tikzpicture}
\caption{Truncation of $X$ in the level $j$.}
\label{Truncation--X}
\end{figure}
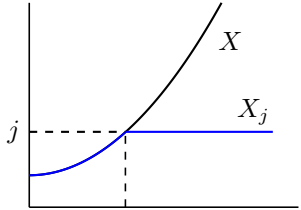

By virtue of the previous result in $L^{p}$, for each $j \in \mathbb{N}$ there exists a unique $\xi_j \in L^{p}(\Omega,\mathcal{G},\mathbb{P})$ satisfying
$$
\int_{\Omega} \xi_j\, \mathbf{1}_A\, d\mathbb{P} 
= \int_{\Omega} X_j\, \mathbf{1}_A\, d\mathbb{P},
\qquad \forall\, A \in \mathcal{G}.
$$

Since $X_j \leq X_{j+1}$, the monotonicity of conditional expectation in $L^{p}$ implies that $\xi_j \leq \xi_{j+1}$ almost surely on $\Omega$. Define
$$
\xi(\omega) := \lim_{j \to \infty} \xi_j(\omega), \qquad \omega \in \Omega,
$$
which is $\mathcal{G}$-measurable because it is the pointwise limit of $\mathcal{G}$-measurable functions.  

Applying the monotone convergence theorem (see \cite[Theorem 9.1]{JacodProtter2004}), we obtain that for each $A \in \mathcal{G}$,
$$
\int_{\Omega} \xi\, \mathbf{1}_A\, d\mathbb{P}
= \lim_{j \to \infty} \int_{\Omega} \xi_j\, \mathbf{1}_A\, d\mathbb{P}
= \lim_{j \to \infty} \int_{\Omega} X_j\, \mathbf{1}_A\, d\mathbb{P}
= \int_{\Omega} X\, \mathbf{1}_A\, d\mathbb{P}.
$$

Hence $\xi = \mathbb{E}(X \mid \mathcal{G})$ for every $X \geq 0$.

For a general random variable $X \in L^1(\Omega,\mathcal{F},\mathbb{P})$, write $X = X^{+} - X^{-}$ with $X^{+}, X^{-} \in L^{1}(\Omega,\mathcal{F},\mathbb{P})$ and $X^{+}, X^{-} \geq 0$. Applying the previous argument to each function, we define
$$
\mathbb{E}(X \mid \mathcal{G}) 
:= \mathbb{E}(X^{+} \mid \mathcal{G}) - \mathbb{E}(X^{-} \mid \mathcal{G}),
$$
which completes the proof.
\end{proof}

In conclusion, formulating conditional expectation through the Riesz representation theorem in $L^{p}$ reveals its genuine analytic nature. 

Rather than an average defined case by case, $\mathbb{E}(X\mid \mathcal{G})$ is the unique element of $L^{p}(\Omega,\mathcal{G},\mathbb{P})$ that reproduces, on the subspace of $\mathcal{G}$-measurable functions, the linear action of $X$ through integration.

\subsection{The Riesz--Markov--Kakutani Theorem: Construction of the Wiener Measure}

The Wiener process, also known as Brownian motion, occupies a central place in the theory of stochastic processes and in the development of stochastic integration. Besides being one of the first stochastic processes to be studied rigorously, it stands out for its structural simplicity and for possessing properties that make it the ideal model for describing continuous random behavior. In fact, the definition of the Itô integral attains its most natural and intuitive form precisely when the integrator is a Wiener process.

From a mathematical point of view, a stochastic process can be completely described through its probability distribution, that is, by a measure assigning probabilities to sets of possible trajectories. In the case of the Wiener process, this law is known as the Wiener measure. This measure acts not on real numbers but on functions: each trajectory is a continuous function of time.

Constructing this measure rigorously requires endowing the space of trajectories with a probabilistic structure consistent with the Gaussian dynamics that characterize the Wiener process. However, this goal cannot be achieved by the elementary methods of finite-dimensional measure theory, since the space of continuous trajectories is not Euclidean. This is precisely where the Riesz--Markov--Kakutani theorem plays an essential role. The theorem establishes a correspondence between continuous linear functionals defined on the space of continuous functions $\mathcal{C}_{c}(\Omega)$ and regular Borel measures on the same space $\Omega$. In other words, constructing a measure amounts to defining a positive and continuous linear functional. Applied to our context, this principle provides a functional approach to obtain the Wiener measure: instead of postulating its existence, we deduce it as the measure representing ---in the sense of Riesz--Markov--Kakutani--- a linear functional built from Gaussian kernels satisfying Kolmogorov’s compatibility equation; cf. \cite[Chapter 6]{Takhtajan2008}.

In this way, the Wiener measure appears as a natural bridge between probability theory and functional analysis. The language of linear functionals not only guarantees its existence and uniqueness; it also provides a precise conceptual interpretation: the Wiener measure represents the distribution of all possible trajectories of a continuous Gaussian process.

We start from the Gaussian kernel
$$
\varphi(x-y,t) = \frac{1}{\sqrt{4\pi D t}}\, 
\exp\!\left(-\frac{|x-y|^{2}}{4 D t}\right),
$$
which is the fundamental solution of the heat equation and describes the transition probabilities of a Gaussian process in $\mathbb{R}^{n}$ (see, for example, \cite{Gorostiza2001,DaPrato1999}). This kernel satisfies Kolmogorov’s compatibility equation (see \cite[Chapter 6]{Takhtajan2008})
$$
\int_{\mathbb{R}}\varphi(x-y,t-s)\,\varphi(y-z,s-u)\,dy 
= \varphi(x-z,t-u),
$$
which expresses the temporal consistency of the transition distributions. In other words, the probability of moving from $x$ to $z$ in time $t-u$ can be decomposed as the integral over all intermediate states $y$, reflecting the Markov property of the Wiener process; cf. \cite[Chapter 5]{DaPrato1999}.

The product of these Gaussian measures in $\mathbb{R}^{n}$ yields a measure such that if $\{V_{i}\}_{i=1}^{n}$ is a collection of Borel sets and $\{t_{i}\}_{i=1}^{n}\subset [0,t]$ is a temporal partition, then the measure of
$$
V := 
\big\{
(x_{1},\ldots,x_{n}) : x_{i} \in V_{i}\ \text{for all}\ i
\big\}
$$
is given by
$$
\int_{V_{1}}\cdots \int_{V_{n}}
\prod_{i=1}^{n}
\varphi(x_{i}-x_{i-1},t_{i}-t_{i-1})\;dx_{1}\cdots dx_{n},
$$
where $x_{0}=x$, $x_{n}=y$, $t_{0}=0$ and $t_{n}=t$. This expression represents the joint probability that the events $V_{1},\ldots,V_{n}$ occur at times $t_{1},\ldots,t_{n}$, respectively, under the Gaussian dynamics governed by $\varphi$.

For the Wiener process, however, the families of sets and times must be arbitrarily large. Hence, the Wiener measure cannot be defined on any finite-dimensional space $\mathbb{R}^{n}$; rather, it must live on an infinite-dimensional space of trajectories.

Let $\Omega_{x,y;t}$ be the space of (possibly discontinuous) trajectories joining $x$ and $y$ in the interval $[0,t]$. Given Borel sets $\{V_{i}\}_{i=1}^{N}$ and a temporal partition $\{t_{i}\}_{i=1}^{N}\subset [0,t]$, define
$$
\mathcal{C} :=
\{
\gamma \in \Omega_{x,y;t} : \gamma(t_{i}) \in V_{i}\ \text{  for all }\ i=1,\ldots,N
\}.
$$

These sets represent the events in which the trajectory $\gamma$ takes prescribed values at times $t_{1},\ldots,t_{N}$.

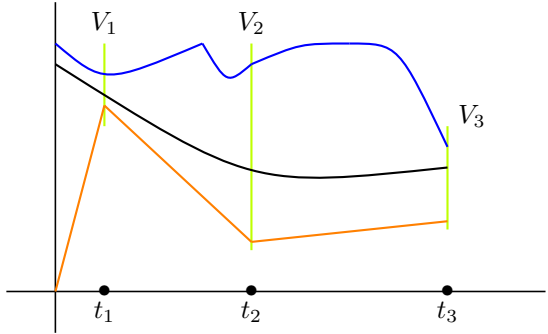
\begin{figure}[!ht]
\centering
\begin{tikzpicture}[xscale=0.65, yscale=0.55]

\draw[-,black, semithick](0,-1) -- (0,7);
\draw[-,black, semithick](-1,0) -- (10,0);

\draw[lime, thick](1,4) -- (1,6);
\draw[lime, thick](4,1) -- (4,6);
\draw[lime, thick](8,1.5) -- (8,4);

\draw (1,0) node{$\bullet$};
\draw (1,-0.5) node{$t_{1}$};
\draw (1,6.5) node{$V_{1}$};

\draw (4,0) node{$\bullet$};
\draw (4,-0.5) node{$t_{2}$};
\draw (4,6.5) node{$V_{2}$};

\draw (8,0) node{$\bullet$};
\draw (8,-0.5) node{$t_{3}$};
\draw (8.5,4.25) node{$V_{3}$};

\draw[orange, thick](0,0) -- (1,4.5);
\draw[orange, thick](1,4.5) -- (4,1.2);
\draw[orange, thick](4,1.2) -- (8,1.7);

\draw[blue, thick](0,6) .. controls (1,5) .. (3,6);
\draw[blue, thick](3,6) .. controls (3.5,5) .. (4,5.5);
\draw[blue, thick](4,5.5) .. controls (5,6) .. (6,6);
\draw[blue, thick](6,6) .. controls (7,6) .. (8,3.5);

\draw[black, thick](0,5.5) .. controls (4,2.5) .. (8,3);

\end{tikzpicture}
\caption{Cylindrical set in the trajectory space $\Omega_{x,y;t}$, with constraints $\gamma(t_i)\in V_i$ and examples of trajectories satisfying them.}
\label{Cyli}
\end{figure}

We want the probability of this event to be
$$
\int_{V_{1}}\cdots \int_{V_{N}}
\prod_{i=1}^{N+1}
\varphi(x_{i}-x_{i-1}, t_{i}-t_{i-1})\;dx_{1}\cdots dx_{N}.
$$

To establish the existence of such a measure via the Riesz--Markov--Kakutani theorem, we must first define a linear functional on a compact Hausdorff space. Thus, we equip the space of trajectories with a topology satisfying these conditions.

Let $\hat{\mathbb{R}} = \mathbb{R}\cup\{\infty\}$ denote the one-point compactification of the real line. For $t>0$, set $\Omega_{t}$ to be the set of all functions $f\colon [0,t]\to \hat{\mathbb{R}}$. Given $x,y\in\mathbb{R}^{n}$, define
$$
\Omega_{x,y;t} := 
\{
\gamma\in\Omega_{t} : \gamma(0)=x,\ \gamma(t)=y
\}.
$$

We endow this space with the product (pointwise convergence) topology $\tau_{\mathrm{prod}}$. By Tychonoff’s theorem, $\Omega_{x,y;t}$ is a compact Hausdorff topological space \cite[Theorem 17.8]{Willard2012}.

Let $\mathcal{B}(\Omega)$ denote the Borel $\sigma$–algebra associated with this topology. Henceforth, the measurable space $\left(\Omega_{x,y;t},\mathcal{B}(\Omega)\right)$ will serve as our underlying domain.

According to the Riesz--Markov--Kakutani theorem (see Theorem \ref{TeoRieszMarkovKakutani}), constructing a regular Borel measure on $\Omega_{x,y;t}$ reduces to defining a positive, bounded linear functional on the space of continuous functions $\mathcal{C}(\Omega_{x,y;t})$. We now introduce the class of functions on which this functional will first be defined.

\begin{definition}
A function $F\colon \Omega_{x,y;t}\to\mathbb{R}$ is called a \textbf{cylindrical function} if $F$ is continuous and there exist a finite set of time points $T=\{0\le t_{1}\le\cdots\le t_{j}\}\subset[0,t]$ and a function $\tilde{F}\colon \hat{\mathbb{R}}^{j}\to\mathbb{R}$ such that
$$
F(\gamma) = \tilde{F}\big(\gamma(t_{1}),\ldots,\gamma(t_{j})\big)
$$
for every $\gamma\in\Omega_{x,y;t}$.
\end{definition}

Denote by $\mathcal{C}_{\rm cyl}(\Omega_{x,y;t})$ the collection of all cylindrical functions. It is a vector subspace of $\mathcal{C}(\Omega_{x,y;t})$, the space of all continuous real-valued functions on $\Omega_{x,y;t}$.

Define the functional $L\colon \mathcal{C}_{\mathrm{cyl}}(\Omega_{x,y;t})\to\mathbb{R}$ by
$$
L(F) =
\int_{\mathbb{R}}\!\cdots\!\int_{\mathbb{R}}
\tilde{F}(x_{1},\ldots,x_{N})
\prod_{i=1}^{N+1}
\varphi(x_{i}-x_{i-1},t_{i}-t_{i-1})
\;dx_{1}\cdots dx_{N},
$$
where $\tilde{F}$ is the function associated with $F$ at the times $(t_{1},\ldots,t_{N})$.

Kolmogorov’s compatibility equation guarantees that $L$ is well defined: if a cylindrical function depends only on $j$ coordinates, extending it artificially to depend on $j+k$ coordinates does not change the value of $L$, since integration against the extra Gaussian kernels collapses exactly to the original expression.

Moreover, since $\varphi$ is nonnegative,
$$
L(1)=\varphi(x-y,t),
$$
and for every $F\in\mathcal{C}_{\mathrm{cyl}}(\Omega_{x,y;t})$,
$$
|L(F)| \le \varphi(x-y,t)\,\|F\|_{\infty}.
$$

Thus, $L$ is a continuous linear functional of norm $\|L\|=\varphi(x-y,t)$.

By the Hahn--Banach theorem \cite[Theorem 1.9.6]{Megginson1998}, $L$ extends to a continuous linear functional on all of $\mathcal{C}(\Omega_{x,y;t})$ without increasing its norm. Hence, we may treat $L$ as an element of the dual space $\mathcal{C}(\Omega_{x,y;t})^{*}$.

Finally, by the Riesz--Markov--Kakutani theorem (see Theorem~\ref{TeoRieszMarkovKakutani}), there exists a unique regular Borel measure $\mathcal{W}^{x,y;t}$ on $\Omega_{x,y;t}$ such that for every cylindrical function $F$,
$$
\begin{aligned}
\int_{\Omega_{x,y;t}} F(\gamma)\, d\mathcal{W}^{x,y;t}(\gamma)
&=\int_{\mathbb{R}}\!\cdots\!\int_{\mathbb{R}}
\tilde{F}(x_{1},\ldots,x_{N})\\
&\qquad\times
\prod_{i=1}^{N+1}
\varphi(x_{i}-x_{i-1},t_{i}-t_{i-1})\,dx_{1}\cdots dx_{N}.
\end{aligned}
$$

We have thus established the following theorem.

\begin{theorem}[Existence of the Wiener Measure]
There exists a unique regular Borel measure $\mathcal{W}^{x,y;t}$ on $\Omega_{x,y;t}$ such that
$$
\begin{aligned}
\int_{\Omega_{x,y;t}} F(\gamma)\, d\mathcal{W}^{x,y;t}(\gamma)
&=\int_{\mathbb{R}}\!\cdots\!\int_{\mathbb{R}}
\tilde{F}(x_{1},\ldots,x_{N})\\
&\qquad\times
\prod_{i=1}^{N+1}
\varphi(x_{i}-x_{i-1},t_{i}-t_{i-1})\,dx_{1}\cdots dx_{N}.
\end{aligned}
$$
for every cylindrical function $F$, where $t_{0}=0$, $t_{N+1}=t$, $x_{0}=x$, and $x_{N+1}=y$.  
The measure $\mathcal{W}^{x,y;t}$ is called the \textbf{Wiener measure} with fixed endpoints $x$ and $y$ at time $t$.
\end{theorem}

This construction not only establishes the existence of the Wiener measure but also provides an effective tool for computing integrals with respect to it. In particular, it allows one to extend the definition of the integral to functions more general than cylindrical ones, while preserving linearity and continuity.

Indeed, let $F$ be a measurable function such that there exists a uniformly bounded sequence of cylindrical functions $(F_{j})_{j\ge1}$ with $F_{j}\to F$ pointwise. By the dominated convergence theorem,
$$
\begin{aligned}
\int_{\Omega_{x,y;t}} F(\gamma)\, d\mathcal{W}^{x,y;t}(\gamma)
&= \lim_{j\to\infty}
\int_{\mathbb{R}}\!\cdots\!\int_{\mathbb{R}}
\tilde{F}_{j}(x_{1},\ldots,x_{N})\\
&\qquad\times
\prod_{i=1}^{N+1}
\varphi(x_{i}-x_{i-1},t_{i}-t_{i-1})\,dx_{1}\cdots dx_{N}.
\end{aligned}
$$

Thus, integrals with respect to the Wiener measure can be approximated by multiple integrals over Euclidean spaces, a feature particularly useful in the study of diffusion processes and in perturbative methods of mathematical physics.

It can also be shown (see, for example, \cite{Nelson1964}) that the set of continuous trajectories $\mathcal{C}_{x,y;t}$ has full measure under $\mathcal{W}^{x,y;t}$, that is,
$$
\mathcal{W}^{x,y;t}(A\cap\mathcal{C}_{x,y;t}) = \mathcal{W}^{x,y;t}(A)
\qquad\text{for all } A\subset\Omega_{x,y;t}.
$$

Consequently, the Wiener measure assigns total probability to continuous trajectories, reflecting the fact that sample paths of the Wiener process are continuous (though nowhere differentiable); see also \cite[Theorem 7.1.6]{Durrett2019}.

This construction highlights the deep connection between functional analysis and probability theory. Starting from a simple linear functional defined on cylindrical functions and relying on the Riesz--Markov--Kakutani representation theorem, one naturally obtains a probability measure on a space of trajectories.

From this perspective, stochastic processes ---and in particular the Wiener process--- appear as manifestations of general principles of duality and continuity in spaces of functions.  
The Wiener measure provides the appropriate analytical framework for the formulation of stochastic integration and for the study of diffusion phenomena, random fluctuations, and stochastic differential equations; cf. \cite{DaPrato1999}.

\bibliographystyle{maa}

\end{document}